\documentclass[oneside]{amsart}
\usepackage{pstricks}
\usepackage{amsmath}
\usepackage{amsfonts}
\usepackage{amsthm}
\usepackage{hyperref}
\usepackage{amscd}
\usepackage{color}
\usepackage{enumerate}
\usepackage{subcaption}
\DeclareCaptionSubType[Alph]{figure}

\newcommand{\bF}{\mathbf{F}}

\newcommand{\ZZ}{\mathbb{Z}}

\newcommand{\PP}{\mathbb{P}}

\newcommand{\A}{\mathbb{A}}
\newcommand{\RR}{\mathbb{R}}

\newcommand{\NN}{\mathbb{N}}
\newcommand{\CO}{\mathcal{O}}

\newcommand{\cX}{\mathcal{X}}

\DeclareMathOperator{\spec}{Spec}
\DeclareMathOperator{\proj}{Proj}
\DeclareMathOperator{\conv}{conv}

\newtheorem{thm}{Theorem}[section]
\newtheorem{prop}[thm]{Proposition}
\newtheorem{lemma}[thm]{Lemma}
\newtheorem{cor}[thm]{Corollary}

\theoremstyle{definition}

\newtheorem{rem}[thm]{Remark}

\newtheorem{ex}[thm]{Example}

\title{Fano Schemes of Lines on Toric Surfaces}

\author{Nathan Ilten}
\address{Department of Mathematics, Simon Fraser University,
8888 University Drive, Burnaby BC V5A1S6, Canada}
\email{nilten@sfu.ca}

\newcommand{\figrns}{%
 \psset{unit=.6cm}
 \begin{pspicture}(0,0)(1,6)%

 \psset{linewidth=1pt}%
\psdots(0,0)(1,0)(0,1)(1,1)(0,3)(1,3)(0,4)(0,6)
\psline(1,1.5)(1,0)(0,0)(0,1)(0,1.5)
\psline[linestyle=dotted](0,1.5)(0,2.5)
\psline[linestyle=dotted](1,1.5)(1,2.5)
\psline(0,2.5)(0,4.5)
\psline(1,2.5)(1,3)(.83,3.5)
\psline(.17,5.5)(0,6)(0,5.5)
\psline[linestyle=dotted](0,4.5)(0,5.5)
\psline[linestyle=dotted](.83,3.5)(.17,5.5)
\rput (1,6) {$(0,\alpha)$}
\rput (2,3.5) {$(1,\beta)$}
\rput (-1,0) {$(0,0)$}
\rput (2,0) {$(1,0)$}
\end{pspicture}}

\newcommand{\figinv}{%
 \psset{unit=1cm}
 \begin{pspicture}(-2,0)(3,5)%
 \psset{linewidth=1pt}%
\psframe[framearc=.5,linecolor=gray](-.4,.7)(2.4,1.3)
\psdots[dotstyle=o,dotsize=.3,linecolor=gray](-1,2)
\multips(0,1){4}{\multips(1,0){6}{\psdots(-2,0)}}
\psline(-2,3)(0,0)(1,0)(3,2)
\psline[linestyle=dotted](-2,3)(-3,4.5)
\psline[linestyle=dotted](3,2)(4,4.5)
\rput(-.5,0) {$b$}
\rput(1.5,0) {$c$}
\rput(-2.5,3) {$a$}
\rput(2.5,1) {$d$}
\rput(1,1.6){$\mu_E=3$}
\rput(-.5,2.5){$\gamma(b)=2$}
\rput(2.3,2.5){$\gamma(c)=\infty$}
\end{pspicture}}

\newcommand{\figcubic}{%
 \psset{unit=1cm}
 \begin{pspicture}(-2,0)(1,4)%
 \psset{linewidth=1pt}%
\multips(0,1){4}{\multips(1,0){4}{\psdots(-2,0)}}
\psline(-1,2.4)(0,0)(1,0)(-.5,2.7)
\psline[linestyle=dotted](-1,2.4)(-1.5,3.6)
\psline[linestyle=dotted](-.5,2.7)(-1,3.6)
\rput(-.5,0) {$b$}
\rput(1.5,0) {$c$}
\rput(0,1.2) {$v$}
\rput(-.8,2.7) {$a'$}
\end{pspicture}}

\newcommand{\figdpa}{%
 \psset{unit=.6cm}
 \begin{pspicture}(-2,-2)(2,2)%
 \psset{linewidth=1pt}%
\multips(0,1){3}{\multips(1,0){3}{\psdots(-1,-1)}}
\psline(-1,-1)(-1,0)(0,1)(1,1)(1,0)(0,-1)(-1,-1)
\end{pspicture}}

\newcommand{\figdpb}{%
 \psset{unit=.6cm}
 \begin{pspicture}(-2,-2)(2,2)%
 \psset{linewidth=1pt}%
\multips(0,1){3}{\multips(1,0){3}{\psdots(-1,-1)}}
\psline(-1,-1)(-1,0)(0,1)(1,0)(0,-1)(-1,-1)
\end{pspicture}}

\newcommand{\figdpc}{%
 \psset{unit=.6cm}
 \begin{pspicture}(-2,-2)(2,2)%
 \psset{linewidth=1pt}%
\multips(0,1){3}{\multips(1,0){3}{\psdots(-1,-1)}}
\psline(-1,0)(0,1)(1,0)(0,-1)(-1,0)
\end{pspicture}}

\newcommand{\figdpd}{%
 \psset{unit=.6cm}
 \begin{pspicture}(-2,-2)(2,2)%
 \psset{linewidth=1pt}%
\multips(0,1){3}{\multips(1,0){3}{\psdots(-1,-1)}}
\psline(-1,0)(0,-1)(1,1)(-1,0)
\end{pspicture}}

\newcommand{\figdpe}{%
 \psset{unit=.6cm}
 \begin{pspicture}(-2,-2)(2,2)%
 \psset{linewidth=1pt}%
\multips(0,1){4}{\multips(1,0){4}{\psdots(-1,-1)}}
\psline(-1,-1)(0,2)(2,0)(-1,-1)
\end{pspicture}}

\begin{document}

\begin{abstract}
We completely describe the Fano scheme of lines $\bF_1(X)$ for a projective toric surface $X$ in terms of the geometry of the corresponding lattice polygon.
\end{abstract}
\maketitle

\section{Introduction}
Let $X$ be a projective variety embedded in $\PP^n$. The Fano scheme of lines $\bF_1(X)$ is the fine moduli space parametrizing lines of $\PP^n$ contained in $X$. Such Fano schemes have been studied extensively in the case of hypersurfaces (e.g. \cite{barth:78a} \cite{harris:98a} \cite{beheshti:06a}), beginning with the classical theorem of Cayley-Salmon that a non-singular cubic surface contains exactly 27 lines.

In this short note, we completely describe the Fano scheme of lines $\bF_1(X)$ for a projective toric surface $X$.
More specifically, let $X\subset \PP^n$ be a normal toric surface embedded in $\PP^n$ by the complete linear system of a very ample divisor $D$. The data of $X$ and $D$ corresponds to a lattice polygon $P\subset \RR^2$, see e.g \cite{fulton:93a}. In our main result (Theorem \ref{thm:main}), we completely describe $\bF_1(X)$ in terms of the geometry of the polygon $P$.

In \S \ref{sec:prelim} we introduce necessary background and state our main result, which we then prove in \S \ref{sec:proof}. As an easy consequence, we deduce that if $X$ is a non-singular toric surface, then $\bF_1(X)$ is also non-singular (Corollary \ref{cor:reduced}).
We conclude in \S \ref{sec:ex} by considering examples as well as applications of our result to non-toric surfaces. Indeed, our result coupled with degeneration techniques may be used to give upper bounds on the number of lines on certain non-toric surfaces.

\section{Preliminaries and Main Result}\label{sec:prelim}
\subsection{Fano schemes}
Let $K$ be an algebraically closed field and $X\subset \PP_K^n$ any projective $K$-scheme. For any natural number $k\in \NN$, the projective scheme $\bF_k(X)$ is the fine moduli space parametrizing $k$-planes of $\PP^n$ contained in $X$, see e.g. \cite[\S IV.3]{eisenbud:00a} for a detailed description. If $T$ is an algebraic torus acting on $\PP^n$ which fixes $X$, then the Fano schemes $\bF_k(X)$ inherit a natural $T$-action. 

\subsection{Projective toric varieties}
Consider an $m$-dimensional polytope $P\subset \RR^m$ whose vertices are lattice points in $\ZZ^m$. Set $$
X_P=\proj K[S_P]
$$
where $S_P$ is the subsemigroup of $\ZZ^m\times \ZZ$ generated by $(P\cap \ZZ^m)\times 1$. Then $X_P$ is an $m$-dimensional projective toric variety (with action by the torus $T=\spec K[\ZZ^m]$) and any normal projective toric variety embedded by a complete linear system is of this form, see e.g. \cite[\S 3.4]{fulton:93a}. In the case of primary interest to us ($m=2$), the converse is true, that is, $X_P$ is always normal and embedded by a complete linear system. 

The torus fixed points of the Fano schemes $\bF_k(X_P)$ are easy to describe. By a \emph{primitive $k$-simplex} we mean the convex hull of $k+1$ lattice points in $\ZZ^m$ containing no other lattice points.
\begin{prop}\label{prop:fixed}
The torus fixed points of $\bF_k(X_P)$ are in bijection with the set of faces of $P$ which are primitive $k$-simplices. 
\end{prop}
\begin{proof}
Any fixed point of $\bF_k(X_P)$ must be the closure of a $k$-dimensional $T$-orbit in $X_P$, which are in bijection to the set of $k$-dimensional faces of $P$ \cite[\S 3]{fulton:93a}. Given a face $Q\subset P$, the corresponding orbit closure as a subvariety of $\PP^n$ is simply $X_Q$ sitting inside of a linear subspace of $\PP^n$. But $X_Q$ is a linear space if and only if there are no relations among the lattice points of $Q$, that is, $Q$ is a primitive $k$-simplex.
\end{proof}

\subsection{Rational normal scrolls}
We say that a toric surface $X_P$ is a \emph{rational normal scroll} if $P$ is lattice equivalent to a polyhedron of the form
$$
P_{\alpha,\beta}=\conv\{(0,0),(1,0),(0,\alpha),(1,\beta)\}\qquad \alpha,\beta\in\ZZ_{\geq 0}\quad a>0\quad \alpha\geq \beta,
$$
see Figure \ref{fig:scroll}. Note that 
$X_{P_{\alpha,\beta}}$ is embedded in $\PP^{\alpha+\beta+1}$ with coordinates $x_u$, $u\in \ZZ^2\cap P$ and cut out by the $2\times 2$ minors of the matrix
\begin{equation}\label{eqn:minors}
\left(
\begin{array}{c c c c c c c c}
x_{(0,0)}&x_{(0,1)}&\ldots&x_{(0,\alpha-1)}&x_{(1,0)}&x_{(1,1)}&\ldots&x_{(1,\beta-1)}\\
x_{(0,1)}&x_{(0,2)}&\ldots&x_{(0,\alpha)}&x_{(1,1)}&x_{(1,2)}&\ldots&x_{(1,\beta)}\\
\end{array}
\right).
\end{equation}

\begin{figure}
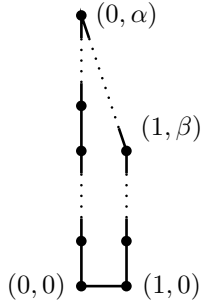

\figrns
\caption{The polytope $P_{\alpha,\beta}$ for a rational normal scroll}\label{fig:scroll}
\end{figure}

Note that as an abstract variety, $X_{P_{\alpha,0}}$ is the (weighted) projective space $\PP(1,1,\alpha)$, and for $\beta\geq 1$, $X_{P_{\alpha,\beta}}$ is the Hirzebruch surface 
$$\mathcal{F}_{\alpha-\beta}=\proj_{\PP^1}\left( \CO_{\PP^1}\oplus\CO_{\PP^1}(\alpha-\beta)\right).$$

\subsection{Lattice geometry}
Let $v,w\in \ZZ^2$ be primitive vectors. Then there exists a unique $p,q\in \ZZ$, $0\leq q< p$ with $q,p$ relatively prime such that $A(v)=(1,0)$ and $A(w)=(-q,p)$ for some $A\in GL_2(\ZZ)$ \cite[\S 2.2]{fulton:93a}. 
Concretely, $p=|\det(v,w)|$ and if $\alpha,\beta$ are integers such that $\alpha v_1+\beta v_2=1$, then $q$ is the $p$ less the remainder of $\alpha w_1+\beta w_2$ by $p$.
We define
$$\gamma(v,w):=\lceil p/q \rceil.$$
Note that if $q=0$, we adopt the convention that $\gamma(v,w)=\infty$.

Now let $P$ be a lattice polygon in $\RR^2$. By Proposition \ref{prop:fixed}, we see that fixed points of $\bF_1(X_P)$ correspond to primitive edges of $P$. For each primitive edge $E=\overline{bc}$ of $P$, we will definite invariants  $\gamma_{b}$, $\gamma_{c}$, and $\mu_E$.
Indeed, let $a,d$ be the other lattice points on the boundary of $P$ adjacent to $b$ and $c$ respectively. Then 
\begin{align*}
\gamma_{b}:&=\gamma(c-b,a-b);\\
\gamma_{c}:&=\gamma(b-c,d-c).
\end{align*}
Likewise, let $u_E\in (\ZZ^2)^*$ be the primitive functional such that $u_E(b)=u_E(c)$, and $u_E(v)\geq u_E(b)$ for all $v\in P$. We define
$$
\mu_E:=\#\{v\in P\cap \ZZ^2\ |\ u_E(v)=u_E(b)+1\}.
$$
Note that $\mu_E\geq 1$.

These quantities have easy geometric interpretations, see Figure \ref{fig:invariants} for an illustration.
The quantity $\mu_E$ is the number of lattice points of $P$ in ``height one'' above the edge $E$. On the other hand, $\gamma(b)$ is the height in which a lattice point of $P$ first appears ``to the left'' of the ray extending from $b$ through the left-most lattice point in height one. Likewise,  $\gamma(c)$ is the height in which a lattice point of $P$ first appears ``to the right'' of the ray extending from $b$ through the right-most lattice point in height one.

\begin{figure}
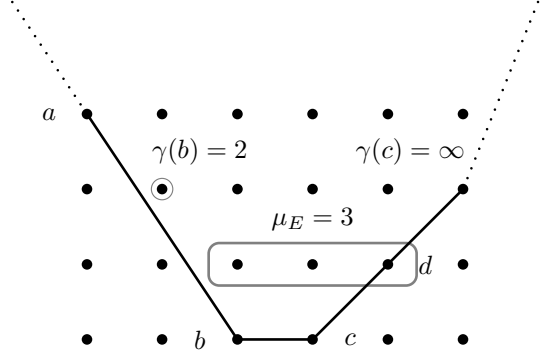

\figinv
\caption{The quantities $\mu_E$, $\gamma(b)$, and $\gamma(c)$}\label{fig:invariants}
\end{figure}

\subsection{Main result}
We now are able to formulate our main result. For any $l\in\NN$, set
$$
\xi_l:=\spec K[z]/z^l.
$$

\begin{thm}\label{thm:main}
Let $X_P$ be a projective normal toric surface.
\begin{enumerate}
\item \label{part:one} If $P=P_{\alpha,\beta}$, then 
$$\bF_1(X_P)=
\begin{cases}
\PP^2 & \alpha=1,\beta=0\\
\PP^1\sqcup \PP^1 & \alpha=\beta=1\\
\PP^1\times \xi_2 & \alpha\geq 2, \beta=0\\
\PP^1\sqcup \xi_1 & \alpha\geq 2, \beta=1\\
\PP^1 & \alpha\geq 2, \beta\geq 2
\end{cases}
$$
\item \label{part:two} If $X$ is not a rational normal scroll, then $\bF_1(X_P)$ is either zero-dimensional or empty, with irreducible components in bijection to the primitive edges $E$ of $P$. For each such edge $E=\overline{bc}$, the corresponding irreducible component $Z_E$ has the form  
$$Z_E=
\begin{cases}
 \xi_1 & \mu_E\geq 2\\
\xi_{\min \{\gamma(b),\gamma(c)\}} & \mu_E=2\\
\xi_{\gamma(b)}\times \xi_{\gamma(c)} & \mu_E=1
\end{cases}.
$$
In particular, we have 
$$\deg Z_E=
\begin{cases}
 1 & \mu_E\geq 2\\
\min \{\gamma(b),\gamma(c)\} & \mu_E=2\\
\gamma(b)\cdot \gamma(c) & \mu_E=1
\end{cases}.
$$
\end{enumerate}
\end{thm}
\begin{rem}
Suppose that $X_P$ is not a rational normal scroll. If $\mu_E=2$, then either $\gamma(b)$ or $\gamma(c)$ must be finite. Likewise, if $\mu_E=1$, then both $\gamma(b)$ and $\gamma(c)$ are finite. Hence, the quantities appearing in part \ref{part:two} of the theorem are finite.
\end{rem}

\begin{cor}\label{cor:reduced}
Let $X_P$ be a non-singular projective toric surface. Then $\bF_1(X_P)$ is non-singular. In particular, it is reduced.
\end{cor}
\begin{proof}
Recall that $X_P$ is non-singular if and only if for every vertex $v$ of $P$, the primitive lattice vectors in the directions of the outgoing edges form a lattice basis \cite[\S 2.1, 3.4]{fulton:93a}. It is straightforward to check that if there is some edge $E$ of $P$ such that  $\mu_E\leq 2$, then $X_P$ must be a rational normal scroll. Furthermore, the case $P= P_{\alpha,\beta}$ for $\alpha\geq 2,\beta=0$ cannot occur, since it is not smooth. The claim now follows directly from Theorem \ref{thm:main}.
\end{proof}
\section{Proof of Main Result}\label{sec:proof}
Throughout this section, we fix a lattice polygon $P$ and consider the toric surface $X_P$.
To each lattice point $u\in\ZZ^2\cap P$ associate a variable $x_u$; we consider the $\ZZ^2$-graded ring
$$
R=K[x_u\ |\ u\in\ZZ^2\cap P]
$$
with grading given by $\deg x_u=u$.
Then $X_P$ is embedded in $\proj R$; its ideal $I_P$ is generated by binomials corresponding to homogeneous relations among the lattice points of $P$, and is homogeneous with respect to the $\ZZ^2$-grading of $R$.

We will begin the proof by describing the Fano scheme in an affine neighborhood of a given toric fixed point of $\bF_1(X_P)$. By Proposition \ref{prop:fixed}, such a fixed point corresponds to a primitive edge $E=\overline{bc}$ of $P$, which we now fix. It will be useful to apply a lattice transformation to bring $P$ into a standard form.
Let $a$ and $d$ be the other lattice points on the boundary of $P$ adjacent to $b$ and $c$, respectively. After applying an invertible affine lattice transformation to $\ZZ^2$, we may assume that $b=(0,0)$, $c=(1,0)$, $a=(-q,p)$ for $p>0$, $q\geq 0$, $p>q$ and $q,p$ relatively prime, arriving at something similar to what is pictured in  Figure \ref{fig:invariants}. Note that now $\mu_E$ is simply equal to 
$$
\mu_E=1+\max \{\lambda \in \ZZ\ |\ (\lambda,1)\in P\}
$$
and $\gamma(b)=\lceil p/q \rceil$.

The line $L_E$ on $X_P$ corresponding to $E$ is given by the vanishing of the coordinates $x_u$, where $u=(u_1,u_2)$ satisfies $u\in P\cap \ZZ^2$ and $u_2>0$. 
We recall the description of $\bF_1(X_P)$ in a neighborhood of this line, see e.g. \cite[\S IV.3]{eisenbud:00a}. 
Let $S$ be the $\ZZ$-graded ring 
$$
S=K[s,t,\sigma_u,\tau_u\ |\ u\in P\cap\ZZ^2,\ u\neq b,c]
$$
with grading given by $\deg s=\deg t=0$, $\deg \sigma_u=\deg \tau_u=u_2$ for $u=(u_1,u_2)$.
We may view the coordinates $\sigma_u,\tau_u$ as parametrizing an open subset of lines in $\proj R$ by considering the line between the points $(x_b=1,x_c=0,x_u=\sigma_u)$ and $(x_b=0,x_c=1,x_u=\tau_u)$.

Consider the homomorphism 
\begin{align*}
\phi:R & \to S\\
x_u &\mapsto \begin{cases}
s&u=b\\
t&u=c\\
\sigma_us+\tau_ut&\mathrm{else}
\end{cases}.
\end{align*}
Note that this maps homogeneous elements of degree $(i,j)$ to elements of degree $j$.
In a neighborhood of $L_E$, $\bF_1(X_P)$ has coordinates $\sigma_u$, $\tau_u$ for $u\neq b,c$ and is cut out by the equations imposed by the condition that $\phi(I_P)=0$. Let $$J\subset K[\sigma_u,\tau_u\ |\ u\neq b,c]$$ denote the corresponding ideal. In other words, $J$ is generated by the coefficients of elements of $f\in \phi(I_P)$ when $f$ is viewed as a polynomial in $s$ and $t$.

Set $v=(0,1)$; note that $v\in P$, since it is in the convex hull of $a,b,c$.
\begin{lemma}\label{lemma:heightone}
Consider $u=(i,j)\in P\cap \ZZ^2$ with $j>0$. Then modulo $J$, $\sigma_u$ and $\tau_u$ can be written in terms of $\sigma_v$ and $\tau_v$. More precisely, we have
\begin{align*}
\sigma_u&=\begin{cases}
0 \mod J & i > 0\\
{j \choose i+j} \sigma_v^{i+j}\tau_v^{-i}\mod J & i\leq 0
\end{cases}\\
\tau_u&=\begin{cases}
0 \mod J & i > 1\\
{j \choose i+j-1} \sigma_v^{i+j-1}\tau_v^{-i+1}\mod J & i\leq 1
\end{cases}.
\end{align*}
\end{lemma}
\begin{proof}
Note that $u=ic+jv$ with $j\geq 1$, $i+j>0$. This translates to the relations 
\begin{align*}
&x_ux_b^{i+j-1}-x_v^jx_c^i\in I_P \qquad &i\geq 0; \\
&x_ux_b^{i+j-1}x_c^{-i}-x_v^j\in I_P \qquad &i\leq 0.
\end{align*}
By inspecting the images of these polynomials under $\phi$, we arrive at the equations
of the desired form. For example, the coefficient of $s^{i+j}$ in $\phi(x_ux_b^{i+j-1}-x_v^jx_c^i)$ is $\sigma_u-\sigma_v^j$ if $i=0$ and is $\sigma_u$ if $i>0$, leading respectively to the equations $\sigma_u=\sigma_v^j$ and $\sigma_u=0$.
\end{proof}

We now discuss various cases which can occur in our local study of $\bF_1(X_P)$ around $L_E$, based on the value of $\mu_E$. For each case, our task is twofold: firstly, determine some relations among the $\sigma_u,\tau_u$; secondly, show that they generate $J$.

\subsection{Case $\mu_E> 2$}\label{sec:g2}
Fix a primitive edge $E$ of $P$ as above and suppose that $\mu_E\geq 2$. Then from above we have that 
$v=(1,0)$, $w=(1,1)$, and $(2,1)$ are all in $P$. The relation
\begin{equation}\label{eqn:square}
x_bx_w-x_cx_v=0
\end{equation}
implies that $\tau_v=\sigma_w=0,\tau_{w}=\sigma_v$ modulo $J$, while 
the relation
$$
x_bx_{(2,1)}-x_cx_{w}=0
$$
implies $\tau_{w}=0$ modulo $J$. Hence, $\sigma_v=\tau_v=0$ modulo $J$, and by Lemma \ref{lemma:heightone} we may conclude that in $\bF_1(X_P)$, the line $L_E$ lies on a single irreducible component, which is isomorphic to $\xi_1$.

\subsection{Case $\mu_E=2$}
Fix a primitive edge $E$ of $P$ and suppose that $\mu_E=2$. By using Equation \eqref{eqn:square}, we get that modulo $J$, $\sigma_v=\tau_{w}$ and $\tau_v=\sigma_{w}=0$.

Suppose that $\gamma(b)=\gamma(c)=\infty$. Then after possibly interchanging $b$ and $c$, $P$ is equal to some $P_{\alpha,\beta}$ with $\beta \geq 1$.
 By Lemma \ref{lemma:heightone} and the above relations, we have that modulo $J$,
\begin{align*}
\sigma_{(0,i)}=\tau_{(1,i)}=\sigma_v^i\\
\tau_{(0,i)}=\sigma_{(1,i)}=0.
\end{align*}
However, these relations generate $J$, since the $2\times 2$ minors of 
\begin{equation*}
\left(
\begin{array}{c c c c c c c c}
\sigma_v^0 s&\sigma_v^1 s&\ldots&\sigma_v^{\alpha-1} s&\sigma_v^0 t&\sigma_v^1 t&\ldots&\sigma_v^{\beta-1} t\\
\sigma_v^1 s&\sigma_v^2 s&\ldots&\sigma_v^\alpha s&\sigma_v^1 t&\sigma_v^2 t&\ldots&\sigma_v^\beta t\\
\end{array}
\right)
\end{equation*}
vanish, and $I_P$ is generated by the $2\times 2$ minors of \eqref{eqn:minors}. Hence in an affine neighborhood of the line $L_E$, $\bF_1(X_P)$ is just the affine line $\A^1$.

Suppose now that $\gamma(b)\neq \infty$.
 Let $a'=(-1,\gamma(b))$. The relation
\begin{equation}\label{eqn:an}
x_{a'}x_b^{\gamma(b)-2}x_c-x_v^{\gamma(b)}=0
\end{equation}
imposes the condition $\sigma_v^{\gamma(b)}=0$ modulo $J$. If $\gamma(c)\neq \infty$, we similarly obtain $\tau_w^{\gamma(c)}=0$ modulo $J$. 
Coupled with the equations of Lemma \ref{lemma:heightone},  we obtain the following equations modulo $J$:
\begin{align}\label{eqn:m2}
\sigma_{(u_1,u_2)}&=\begin{cases}
\sigma_v^{u_2} & u_1=0\ \textrm{and}\  u_2 < \gamma \\
0 &\ \textrm{else}
\end{cases}\\
\label{eqn:m2a}\tau_{(u_1,u_2)}&=\begin{cases}
\sigma_v^{u_2} & u_1=1\ \textrm{and}\  u_2 < \gamma \\
0 &\ \textrm{else}
\end{cases}\\
\sigma_v^\gamma&=0.\label{eqn:m2b}
\end{align}
where $\gamma=\min \{\gamma(b),\gamma(c)\}$.

We claim that these relations generate the ideal $J$. Indeed, the ideal $I_P\subset R$ is generated by homogeneous binomials $f=f_1-f_2$. Consider such a binomial $f$ of degree $(i,j)$. Then $\deg \phi(f)=j$, and if $j \geq \gamma$, then this must vanish modulo the relations of \eqref{eqn:m2},  \eqref{eqn:m2a}, and \eqref{eqn:m2b}.
On the other hand, if $j<\gamma$, then the only coordinates $x_u$ appearing in $f$ must satisfy  $u \in P_{\gamma-1,\gamma-1}$ and the argument used above for $P_{\alpha,\beta}$ shows that $\phi(f)$ vanishes modulo the above relations.

Thus, if not both $\gamma(b)$ and $\gamma(c)$ are infinite, we may conclude that $L_E$ lies on a single irreducible component of $\bF_1(X_P)$, which is isomorphic to $\xi_\gamma$. 

\subsection{Case $\mu_E=1$}
We first analyze the possible values of $\gamma(b)$ and $\gamma(c)$ under the assumption that $\mu_E=1$. Without loss of generality, we assume that $\gamma(b)\geq \gamma(c)$. If $\gamma(b)=\infty$, then we have $P=P_{\alpha,0}$ for some $\alpha\geq 1$. Note that $\gamma(c)=2$ unless $\alpha=1$, in which case $X_P$ is $\PP^2$ in its linear embedding and $\bF_1(X_P)$ is simply $(\PP^2)^*\cong \PP^2$.

 Suppose instead that $\gamma(b)$ is finite. Note that by definition of $\gamma(c)$, $P$ must contain a lattice point of the form
$(2-\lambda\gamma(c),\gamma(c))$ for some $\lambda\in \ZZ$. But since $\mu_E=1$, this implies that $\lambda \geq 1$. On the other hand, we must also have
$$\gamma(c)q\geq p(\lambda\gamma(c)-2)$$
which in turn implies
$$\gamma(c)> (\gamma(b)-1)(\gamma(c)-2),$$
that is,
$$2(\gamma(b)+\gamma(c)-1)>\gamma(b)\gamma(c).$$
This is only possible if $\gamma(c)=2$ or $\gamma(b)=\gamma(c)=3$.

We now consider the case \fbox{$\gamma(c)=2$.} Then $d'=(0,2)\in P$, and we have the relation
$$
x_bx_{d'}-x_v^2\in I_P.
$$
This, together with Lemma \ref{lemma:heightone}, leads to the relations
\begin{align}
\label{eqn:b21}&\sigma_{(0,j)}=\sigma_v^j,\quad \tau_{(0,j)}=j\tau_v\sigma_v^{j-1}\qquad (0,j)\in P\\
\label{eqn:b22}&\tau_v^2=0
\end{align}
modulo $J$.
Now, note that modulo the relation $\tau_v^2=0$, the $2\times 2$ minors of the matrix
\begin{equation}\label{eqn:minors2}
\left(
\begin{array}{c c c c c}
\sigma_v^0s & \sigma_v^1s+\tau_vt & \sigma_v^2s+2\sigma_v\tau_vt& \ldots & \sigma_v^{\alpha-1}s+(\alpha-1)\sigma_v^{\alpha-2}\tau_vt\\
\sigma_v^1s+\tau_vt & \sigma_v^2s+2\sigma_v\tau_vt& \sigma_v^3s+3\sigma_v^2\tau_vt&\ldots & \sigma_v^{\alpha}s+\alpha\sigma_v^{\alpha-1}\tau_vt\\
\end{array}
\right)
\end{equation}
all vanish.
Hence, if $P=P_{\alpha,0}$ $\alpha \geq 2$, the relations of \eqref{eqn:b21} and \eqref{eqn:b22} generate $J$, and we conclude that in a neighborhood of $L_E$, $\bF_1(X_P)$ is just $\A^1\times\xi_2$.

Suppose that $P\neq P_{\alpha,0}$ $\alpha \geq 2$; then $\gamma(b)$ is finite. We can then use Equation \eqref{eqn:an} to obtain the relations $\sigma_v^{\gamma(b)}=0$ and $\sigma_{a'}=\gamma(b)\sigma_v^{\gamma(b)-1}\tau_v$ modulo $J$. With Lemma \ref{lemma:heightone}, this implies that $\sigma_u=\tau_u=0$ modulo $J$ for $u=(u_1,u_2)$ with $u_2>\gamma(b)$.

We claim that these relations generate all of $J$. Indeed, consider a homogeneous binomial $f$ of degree $(i,j)$ in $I_P$. Arguing as in the case $\mu_E=2$, if
$\deg f>j$, then $\phi(f)$ will vanish modulo the above relations.
Suppose instead that $\deg f \leq j$.
If the only coordinates appearing in $f$ are contained in some $P_{\alpha,0}$, then the vanishing of the minors of \eqref{eqn:minors2} shows that $\phi(f)$ must vanish modulo the relations. The only possible monomial which can appear in such $f$ which is not in $P_{\alpha,0}$ is $a'$, and any $f$ containing it must have the form $f=f_1-f_2$, where $f_1=x_{a'}x_b^{\gamma(b)}x_c$ and $f_2$ only involves coordinates in $P_{\alpha,0}$.  But modulo equations in $I_P$ only involving coordinates in $P_{\alpha,0}$, such a binomial is equal to the left hand side of \eqref{eqn:an}, whose image under $\phi$ vanishes modulo the above relations.

Hence, these relations generate $J$, and we may conclude that the point $L_E$ of $\bF_1(X_P)$ lies on a single component, which is isomorphic to $\xi_{\gamma(b)}\times\xi_2$.

\begin{figure}
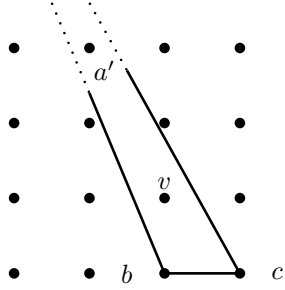

\figcubic
\caption{The case $\mu_E=1$, $\gamma(b)=\gamma(c)=3$}\label{fig:33}
\end{figure}

We now deal with the remaining case \fbox{$\gamma(b)=\gamma(c)=3$.} Then $P$ contains the lattice points $b=(0,0)$, $c=(1,0)$, $v=(0,1)$, and $a'=(-1,3)$, and these are the only lattice points $u=(u_1,u_2)$ contained in $P$ with $u_2\leq 3$, see Figure \ref{fig:33}. Now, any relation among these four lattice points is a multiple of the cubic
$$
x_v^3-x_a'x_bx_c
$$
which imposes exactly the conditions
\begin{align*}
\sigma_v^3&=0\ \qquad& \tau_v^3&=0\\
\sigma_{a'}&=3\sigma_v^2\tau\qquad &\tau_{a'}=&3\sigma\tau^2.
\end{align*}
Lemma \ref{lemma:heightone} implies that for any $u\in P$ with $u_2>3$, we must have $\sigma_u=\tau_u=0$. Now, arguments similar to those used in the previous cases imply that $J$ is generated exactly by the above relations, so we may conclude that the point $L_E$ of $\bF_1(X_P)$ lies on a single component, which is isomorphic to $\xi_3\times\xi_3$.

\subsection{Conclusion of the proof}
We now use our local study to completely describe $\bF_1(X_P)$. A key point in our argument is that since $\bF_1(X_P)$ is projective, any irreducible component of $\bF_1(X_P)$ must contain a fixed point, that is, a point corresponding to one of the lines $L_E$. Thus, by the above local study, we already know the local structure of every component of our Fano scheme.

First, suppose that $X_P$ is a rational normal scroll, that is, $P=P_{(\alpha,\beta)}$. As noted above, if $\alpha=1,\beta=0$, then $X_P$ is just $\PP^2$, and its Fano scheme of lines is $\PP^2$ as well. If $\alpha=\beta=1$, $P$ has four primitive edges $E_1,\ldots,E_4$, and each has $\mu_{E_i}=2$ and the irreducible component containing $L_{E_i}$ is locally just $\A^1$. The only possibility for gluing four copies of $\A^1$  to get a projective scheme results in $\PP^1\sqcup \PP^1$; in our case, this gluing can be seen quite explicitly. Alternatively, if $\alpha=\beta=1$, then $X_P$ is a nonsingular quadric in $\PP^3$, isomorphic to $\PP^1\times\PP^1$, and it is well known that $\bF_1(X_P)\cong \PP^1\sqcup\PP^1$, the two irreducible components corresponding to the two different rulings coming from projecting to the first and second $\PP^1$ factors.

If $\alpha\geq 2$ and $\beta=0$, $P$ has two primitive edges $E_1,E_2$, for which we have $\mu_{E_1}=\mu_{E_2}=1$. The local analysis shows that our Fano scheme is constructed by gluing two copies of $\A^1\times\xi_2$, resulting in $\bF_1(X_P)\cong \PP^1\times\xi_2$. If $\alpha\geq 2$ and $\beta=1$, $P$ has three primitive edges, one with $\mu=\alpha+1$ and the other two with $\mu=2$. For the edges with $\mu=2$ we get two copies of $\A^1$ gluing to a $\PP^1$, and for the edge with $\mu=\alpha+1>2$ we get an isolated point; hence $\bF_1(X_P)\cong \PP^1\sqcup \xi_1$. If on the other hand, If $\alpha,\beta\geq 2$, then $P$ has only two primitive edges, both with $\mu=2$, and  we get obtain two copies of $\A^1$ gluing to give $\bF_1(X_P)\cong \PP^1$.

If we assume instead that $X_P$ is not a rational normal scroll, our local analysis shows that set theoretically, $\bF_1(X_P)$ consists only of a finite number of fixed points, and the scheme structure is exactly as claimed in the theorem.

This concludes the proof of Theorem \ref{thm:main}.\qed

\section{Applications and Examples}\label{sec:ex}
We illustrate Theorem \ref{thm:main} with several examples.

\begin{figure}
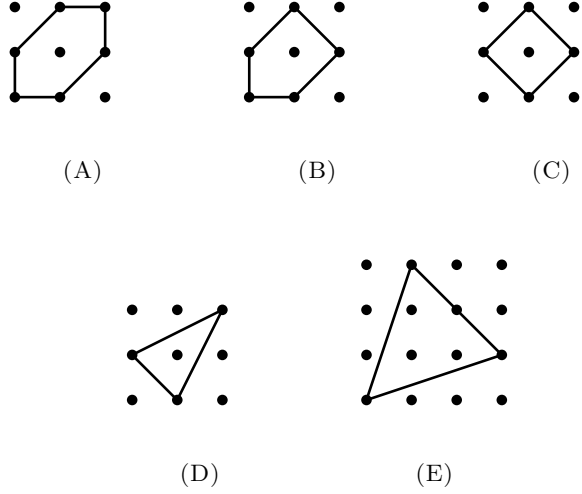

\begin{subfigure}{3cm}
\figdpa
\caption{}\label{fig:a}
\end{subfigure}
\begin{subfigure}{3cm}
\figdpb
\caption{}\label{fig:b}
\end{subfigure}
\begin{subfigure}{3cm}
\figdpc
\caption{}\label{fig:c}
\end{subfigure}

\vspace{1cm}

\begin{subfigure}{3cm}
\figdpd
\caption{}\label{fig:d}
\end{subfigure}
\begin{subfigure}{3cm}
\figdpe
\caption{}\label{fig:e}
\end{subfigure}
\caption{Examples for Theorem \ref{thm:main}}\label{fig:ex}
\end{figure}

\begin{ex}\label{ex:deg}
We consider the toric surfaces corresponding to the polygons pictured in Figure \ref{fig:ex}. For \ref{fig:a}, $X_P$ is the smooth del Pezzo surface of degree $6$ in its anticanonical embedding. Since $P$ has exactly $6$ primitive edges, we see that $\bF_1(X_P)$ consists of six isolated points, and $\deg \bF_1(X_P)=6$.

For \ref{fig:b}, $X_P$ is a singular del Pezzo surface of degree $5$ in its anticanonical embedding. $P$ has five primitive edges; for two we have $\mu_E=3$ giving us two copies of $\xi_1$, for two we have $\mu_E=2$ with $\min\{\gamma(b),\gamma(c)\}=2$ giving us two copies of $\xi_2$, and the remaining edge has $\mu_E=1$ with $\gamma(b)=\gamma(c)=2$, giving us a copy of $\xi_2\times \xi_2$. In total, we obtain $\deg \bF_1(X_P)=10$.

For \ref{fig:c}, $X_P$ is a singular del Pezzo surface of degree $4$ in its anticanonical embedding, a complete intersection of two singular quadrics. $P$ has four primitive edges; each has $\mu_E=1$ and $\gamma(b)=\gamma(c)=2$, giving us four copies of $\xi_2\times \xi_2$. In total, we obtain $\deg \bF_1(X_P)=16$.

For \ref{fig:d}, $X_P$ is a singular del Pezzo surface of degree $3$ in its anticanonical embedding, a singular cubic. $P$ has three primitive edges; each has $\mu_E=1$ and $\gamma(b)=\gamma(c)=3$, giving us three copies of $\xi_3\times \xi_3$. In total, we obtain $\deg \bF_1(X_P)=27$.

For \ref{fig:e}, $X_P$ is a singular del Pezzo surface of degree $2$ in the embedding given by twice the anticanonical class. $P$ has two primitive edges; each has $\mu_E=1$ and $\gamma(b),\gamma(c)=2,3$, giving us two copies of $\xi_2\times \xi_3$. In total, we obtain $\deg \bF_1(X_P)=12$.
\end{ex}

Given any projective scheme $X\subset \PP^n$, its Fano scheme $\bF_k(X)$ embeds as a subscheme of the Grassmannian $G(k+1,n+1)$. By $\deg \bF_k(X)$, we mean the degree of $\bF_k(X)$ obtained by composing the embedding $\bF_k(X)\hookrightarrow G(k+1,n+1)$ with the Pl\"ucker embedding of $G(k+1,n+1)$. Note that if $\bF_k(X)$ is zero-dimensional, the number of $k$-planes contained in $X$ is bounded above by $\deg \bF_k(X)$, with equality if and only if $\bF_k(X)$ is reduced.

\begin{prop}\label{prop:degen}
Let $S$ be a non-singular curve, $\cX\subset \PP^n\times S$ a flat projective family of $K$-schemes over $S$ with general fiber $Y$ and special fiber $X$. Suppose that $\dim \bF_k(X)=\dim \bF_k(Y)$. Then 
$$
\deg \bF_k(Y)\leq \deg \bF_k(X).
$$
\end{prop}
\begin{proof}
The relative Fano scheme $\bF_k(\cX/S)$ gives a family over $S$ whose fibers are the Fano schemes $\bF_k(\cX_s)$, where $\cX_s$ is the fiber of $\cX$ over $s\in S$. In general, this family will not be flat. By considering the special fiber of the closure of $\bF_k(\cX/S)\setminus \bF_k(X)$, we get a subscheme $\bF_k(X)'\subset \bF_k(X)$ fitting into a flat family with $\bF_k(Y)$ \cite[Proposition II.29]{eisenbud:00a}. 
Since $\dim \bF_k(X)' = \dim \bF_k(Y)=\dim \bF_k(X)$, we have $\deg \bF_k(X)' \leq \deg \bF_k(X)$. Finally, $\deg \bF_k(Y) = \deg \bF_k(X)'$ since flatness preserves Hilbert polynomials \cite[Theorem 9.9]{hartshorne:77a}.
\end{proof}

Using the above proposition coupled with Theorem \ref{thm:main}, we may obtain upper bounds on the number of lines on certain non-toric surfaces. We illustrate this application by recovering bounds on the number of lines on certain del Pezzo surfaces.

\begin{ex}
It is classically known that general del Pezzo surfaces of degrees $3$, $4$, and $5$ in their anticanonical embeddings contain $27$, $16$, and $10$ lines, respectively. Our methods reprove that these quantities are upper bounds for the number of lines. Indeed, such general surfaces have degenerations to the surfaces of \ref{fig:d}, \ref{fig:c}, and \ref{fig:b}, respectively, and the degrees of their Fano schemes calculated in Example \ref{ex:deg} are exactly the quantities $27$, $16$, and $10$.

Slightly more subtle arguments show that for any cubic surface $X\subset \PP^3$ with $\dim \bF_1(X)=0$, we have 
$\deg \bF_1(X)=27$,  cf. \cite[Exercise IV.74]{eisenbud:00a}. Indeed, $\bF_1(X)$ is a codimension-four local complete intersection in $G(2,4)$, so for any degeneration as in Proposition \ref{prop:degen}, the relative Fano scheme $\bF_k(\cX/S)$ will be flat over $S$. Likewise, for any complete intersection $X$ of 2 conics in $\PP^4$ with $\dim \bF_1(X)=0$, we have $\deg \bF_1(X)=16$. Indeed, in this case $\bF_1(X)$ is a codimension-six local complete intersection in $G(2,5)$ and a similar argument applies. 
\end{ex}

\subsection*{Acknowledgments}
We thank Andreas Hochenegger for helpful comments.

\bibliographystyle{alpha}
\bibliography{fano-surfaces}
\end{document}